\documentclass[10pt,a4paper,twoside,reqno]{amsart}
\usepackage{amsmath,amsthm,amssymb,amsfonts,xspace}
\usepackage[T1]{fontenc}
\usepackage[utf8]{inputenc}
\usepackage[english]{babel}
\usepackage{amsxtra}
\usepackage{enumerate}
\usepackage{verbatim}
\usepackage{color}
\usepackage[mathscr]{eucal}

\usepackage[bookmarksnumbered,colorlinks]{hyperref}
\usepackage[colorinlistoftodos,italian]{todonotes}
\usepackage[normalem]{ulem}
\usepackage{tensor}

\newcommand{\N}{\mathbb N}

\newcommand{\R}{\mathbb R}

\newtheorem{Theorem}{Theorem}
\newtheorem{Lemma}{Lemma}

\newtheorem{Proposition}{Proposition}
\theoremstyle{definition}
\newtheorem{Remark}{Remark}

\DeclareMathOperator*{\dive}{div}

\def\bal#1\eal{\begin{align}#1\end{align}}              
\def\baln#1\ealn{\begin{align*}#1\end{align*}}          
\def\bml#1\eml{\begin{multline}#1\end{multline}}        
\def\bmln#1\emln{\begin{multline*}#1\end{multline*}}  
\def\bga#1\ega{\begin{gather}#1\end{gather}}
\def\bgan#1\egan{\begin{gather*}#1\end{gather*}}

\newcommand{\beq}{\begin{equation}}
\newcommand{\eeq}{\end{equation}}
\newcommand{\bere}{\begin{Remark}}
\newcommand{\ere}{\end{Remark}}
\newcommand{\bpr}{\begin{Proposition}}
\newcommand{\epr}{\end{Proposition}}

\begin{document}

\title[Harmonic coordinates and regularity  of Berwald  metrics]{Harmonic coordinates for the nonlinear Finsler Laplacian and some regularity results for Berwald metrics}
\author[E. Caponio]{Erasmo Caponio}
\address{Department of Mechanics, Mathematics and Management, \hfill\break\indent
	Politecnico di Bari, Via Orabona 4, 70125, Bari, Italy}
\email{caponio@poliba.it}

\author[A. Masiello]{Antonio Masiello}
\address{Department of Mechanics, Mathematics and Management, \hfill\break\indent
	Politecnico di Bari, Via Orabona 4, 70125, Bari, Italy}
\email{antonio.masiello@poliba.it}


\begin{abstract}
We prove existence of harmonic coordinates for the nonlinear  Laplacian of a Finsler manifold and apply them in a proof of the Myers-Steenrod theorem for Finsler manifolds. Different from the Riemannian case, these coordinates are not suitable for  studying optimal  regularity of the fundamental tensor; nevertheless, we obtain some partial results in this direction when the Finsler metric is Berwald.
\end{abstract}
\maketitle

\section{Introduction}
The existence  of harmonic coordinates 
is well-known  both on Riemannian and Lorentzian manifolds. Actually, apart from isothermal coordinates on a surface, the problem of existence of harmonic coordinates on a Lorentzian manifold ({\em wave harmonic}) was considered before the Riemannian case. Indeed,  A. Einstein himself, T. De Donder and C. Lanczos considered  harmonic coordinates  in the study of  the Cauchy problem for the Einstein field equations,  cf. \cite{Choque15}.
In a Riemannian manifold with a smooth metric, a proof of the existence of harmonic coordinates  is given  in  \cite[Lemma 1.2]{DeTKaz81} but actually   it can be found  in the work of other authors as  \cite{Muelle70} or  \cite[p. 231]{BeJoSc64}. 
The motivation  to consider harmonic coordinates  comes from  fact that the expression of the Ricci tensor in such coordinates simplifies highly (see the introduction in \cite{DeTKaz81}).  

Different from the Riemannian case, the Finsler Laplacian  is a quasi-linear operator and, although it is uniformly elliptic with smooth coefficients where $du\neq 0$, the lack of definition of the coefficients on the set where $du=0$ makes  the analogous Finslerian problem not completely similar to the Riemannian one.   In a 
recent  paper,  T. Liimatainen  and M. Salo \cite{LiiSal14} considered the problem of existence of ``harmonic'' coordinates for  non-linear degenerate elliptic operators on Riemannian manifold,   including the $p$-Laplace operator. 
We will  show that this  result extends to  the Finslerian Laplace operator as well:
\begin{Theorem}\label{main}
	Let  $(M, F, \mu)$ be a smooth Finsler manifold  of dimension $m$ endowed with a smooth volume form $\mu$, such  that $F\in C^{k+1}(TM\setminus 0)$,  $k\geq  2$, (resp. $F\in C^{\infty}(TM\setminus 0)$; $M$ is endowed with an analytic structure, $\mu$ is also analytic and  $F \in C^{\omega}(TM\setminus 0)$). Let  $p\in M$ then  there exists a neighbourhood $V$ of $p$ and a map $\Psi:V\to \Psi(V)\subset \R^m$ such that $\Psi=(u^1,\ldots, u^m)$ is a $C^{k-1,\alpha}$, $\alpha\in(0,1)$ depending on $V$,  (resp. $C^\infty$; analytic) diffeomorphism and  $\Delta u^i=0$, for all $i=1,\ldots, m$, where $\Delta$ is the nonlinear Laplacian  operator associated to $F$.	
\end{Theorem}
Theorem~\ref{main} is proved in Section~\ref{harmonic}, where we also show (Proposition~\ref{MyersSteenrod}) that  harmonic (for the Finsler non-linear Laplacian) coordinates can be used to prove the Myers-Steenrod theorem about regularity of  distance preserving bijection between Finsler manifolds. In the Riemannian case, this was established first by  M. Taylor  in \cite{Taylor06}.

For  a semi-Riemannian metric $h$,  
the expression  in local coordinates of the Ricci tensor is given by: 
\[\big(\mathrm{Ric}(h)\big)_{ij}=-\frac1 2 h^{rs}\frac{\partial^2 h_{ij}}{\partial x^r\partial x^s}+\frac 1 2\left(h_{ri}\frac{\partial H^r}{\partial x^j}+h_{rj}\frac{\partial H^r}{\partial x^i}\right)+\text{\it lower order terms},\]
(see \cite[Lemma 4.1]{DeTKaz81}) where $H^r:=h^{ij}H^r_{ij}$,  and  $H^r_{ij}$ are the Christoffel symbols of $h$. As recalled above, in harmonic coordinates the higher order terms in this expression  simplify to 
$-\frac1 2 h^{rs}\frac{\partial^2 h_{ij}}{\partial x^r\partial x^s}$
because  \beq\label{gammah}H^r=0\eeq
in such  coordinates. When $h$ is  Riemannian,  by   regularity  theory for elliptic  PDEs system  \cite{DouNir55, Morrey58}, this observation leads to optimal regularity results for the components of the metric $h$,  once  a certain level of regularity of the Ricci tensor is known \cite{DeTKaz81}.
Roughly speaking,  the fact that the nonlinear Laplacian is a differential operator on $M$ while the fundamental tensor and the components of any Finslerian  connection are objects defined  on $TM\setminus 0$,   harmonic coordinates for the nonlinear Finsler Laplacian do not give such type of information (see  Remark~\ref{horizontal}).
Nevertheless, in Section \ref{Berwald}, we will consider these type of problems  for Berwald metrics and we will obtain some partial result in this direction.
\section{About the nonlinear Finsler Laplacian}
Let $M$ be a smooth (i.e. $C^{\infty}$), oriented,  manifold of dimension $m$ and let us denote by $TM$ and $TM\setminus 0$, respectively,  the tangent bundle and the slit tangent bundle of $M$, i.e. $TM\setminus 0:= \{v\in TM: v\neq 0\}$. A Finsler metric on $M$ is a non-negative function on $TM$ such that for any $x\in M$, $F(x,\cdot)$ is a strongly convex Minkowski norm on $T_xM$, i.e. 
\begin{itemize}
	\item $F(x,\lambda v)=\lambda F(x,v)$ for all $\lambda >0$ and $v\in TM$, $F(x,v)=0$ if and only if $v=0$;
	\item the bilinear symmetric form on $T_x M$, depending on $(x,v)\in TM\setminus 0$,
	\beq\label{fundtens} g(x,v)[w_1,w_2]:=\!\frac 1 2\frac{\partial^2}{\partial s\partial t}F^2(x, v+sw_1+tw_2)|_{(s,t)= (0,0)}\eeq is positive definite  for all $v\in T_xM\setminus \{0\}$ and it is called the {\em fundamental tensor} of $F$.  
\end{itemize} 
Let us recall  that a function defined in some open subset $U$ of $\R^m$ is of class $C^{k,\alpha}$, for $k\in \N$ and $\alpha\in (0,1]$,  (resp. $C^\infty$; $C^\omega$), if all its derivative up to order $k$ exist and are continuous in $U$ and its $k$-th derivatives are H\"older continuous in $U$ with exponent $\alpha$ (resp. if the derivatives of any order exists and are continuous in $U$; if it is real analytic in $U$).

We  assume that $F \in C^{k+1}(TM\setminus 0)$, where $k\in \N$, $k\geq 2$  (resp. $F \in C^{\infty}(TM\setminus 0)$; $F \in C^{\omega}(TM\setminus 0)$, provided that $M$ is endowed with an analytic structure) in the natural charts of  $TM$ associated to  an atlas of $M$; it is easy  to prove, by using $2$-homogeneity of the function $F^2(x,\cdot)$, that  $F^2$ is $C^1$ on  $TM$ with Lipschitz derivatives on  subsets of $TM$ of the type $K\times \R^m$, with $K$ compact in $M$. 

Let us consider the function 
\[F^*\colon T^*M\to [0,+\infty), \quad\quad F^*(x,\omega):=\max_{\substack{v\in T_xM\\ F(v)=1}}\omega(x)[v],\]
which is  a {\em co-Finsler} metric, i. e. for any  $x\in M$,  $ F^*(x,\omega)$ is a Minkowski norm on $T^*_xM$. 
It is well-known (see, e.g.,  \cite[p. 308]{Shen04}) that $F=F^*\circ \ell$, where $\ell:TM\to T^*M$ is the Legendre map:
\[\ell(x,v)=\left(x,\frac 1 2 \frac{\partial}{\partial v} F^2(x,v)[\cdot]\right),\]
and  $\frac{\partial}{\partial v} F^2(x,v)[\cdot]$ is the vertical derivative of $F^2$ evaluated at $(x,v)$, $\frac{\partial}{\partial v} F^2(x,v)[u]:=\frac{d}{dt} F(v+tu)|_{t=0}$. Thus, $F^*=F\circ\ell^{-1}$, and $F^*\in C^k(T^*M\setminus 0)$ (resp. $F^*\in C^\infty(T^*M\setminus 0)$; $F^*\in C^\omega(T^*M\setminus 0)$). 
Its  fundamental tensor,   obtained as in \eqref{fundtens},  will be denoted by $g^*$.
The components $\left(g^*\right)^{ij}(x,\omega)$ of $g^*(x,\omega)$, $(x,\omega)\in T^*M\setminus 0$,  in natural  local coordinate of $T^*M$ define a square matrix which is the inverse of the one defined by the components  $g_{ij}\big(\ell^{-1}(x,\omega)\big)$ of $g\big(\ell^{-1}(x,\omega)\big)$.

Henceforth, we will often omit the dependence on $x$ in  $F(x,v)$, $F^*(x,\omega)$, $g(x,v)$, $g^*(x,\omega)$, etc.  (which is implicitly carried on by vectors or covectors) writing  simply $F(v)$, $F^*(\omega)$, $g_v$, $g^*_\omega$, etc. 

For a differentiable function    $f\colon  M\to \R$,  the {\em gradient} of $f$ is defined as $\nabla f:=\ell^{-1}(df)$. Hence, $F(\nabla f)=F^*(df)$ and, wherever $df\neq 0$, $df=\ell(\nabla f)=g_{\nabla f}(\nabla f, \cdot)$. 

Given a smooth  volume form $\mu$ on $M$, $\mu$ locally given as $\mu=\sigma dx^1\wedge\ldots\wedge dx^m$, the {\em divergence} of a vector field $x\in \mathfrak X(M)$ is defined as the function $\dive (X)$ such that $\dive (X)\mu=\mathfrak L_X \mu$, where $\mathfrak L$ is the Lie derivative; in local coordinates this is the function 
$\frac{1}{\sigma}\partial_{x^i}(\sigma X^i)$. 
The Finslerian Laplacian of a smooth function on $M$ is then defined as $\Delta f:= \dive(\nabla f)$, thus in local coordinates it is given by
\[\Delta f=\frac{1}{\sigma} \partial_{x^i} \left(\sigma\frac 1 2\frac{\partial F^{*2}}{\partial \omega_i}(df)\right),\]
where $(x^i,\omega_i)_{i=1,\ldots,m}$ are natural local coordinates of $T^*M$ and the Einstein summation convention has been used.
Notice that, wherever $df\neq 0$, $\Delta f$ is equal to 
\[
\Delta f=\frac{1}{\sigma} \partial_{x^i} \left(\sigma (g^*_{df})^{ij}\partial_{x^j} f\right),\] 
thus  $\Delta$ is a quasi-linear operator and when $F$ is the norm of a Riemannian metric $h$ and $\sigma=\sqrt{\det h}$ it becomes linear and equal to the Laplace-Beltrami operator of $h$.

Let $\Omega\subset M $ be an open relatively compact subset  with smooth  boundary and let  $H^k_{\mathrm{loc}}(\Omega)$, $k\in\N\setminus \{0\}$, be   the Sobolev space of  functions defined on $\Omega$ that are  of $H^k$ class on the open subsets $\Omega'$ with compact closure contained in $\Omega$, $\Omega'\subset \subset \Omega$ ($H^k_{\mathrm{loc}}(\Omega)$  can be defined only in terms of the differentiable  structure of $M$, see  \cite[\S 4.7]{Taylor11}); let us also denote by $H^1(\Omega)$ and $H^1_0(\Omega)$ the usual Sobolev spaces on a smooth  compact manifold with boundary (see \cite[\S 4.4--4.5]{Taylor11}).

Let  $E(u):=\frac 1 2\int_\Omega (F^{*2}(du)d\mu=\frac 1 2\int_\Omega F^2(\nabla u)d\mu\in [0,+\infty)$ be  the Dirichlet functional of $(M,F)$. Let $\varphi\in H^1(\Omega)$, then the critical points $u$ of $E$ on $\{\varphi\}+H^1_0(\Omega)$  are the weak solutions of 
\beq\begin{cases}
	\Delta u=0&\text{in } \Omega\\
	u=\varphi&\text{on }\partial\Omega
\end{cases}\label{dirichlet}\eeq
i.e.  for all $\eta\in H^1_0(\Omega)$ it holds:
\[\frac 12 \int_\Omega \frac{\partial F^{*2}}{\partial \omega}(du)[d\eta]d\mu=0, \quad\quad \text{$u-\varphi\in H^1_0(\Omega)$.}\]
Let $(V,\phi)$ be a coordinate system in $M$, with $\phi:V\subset M \to U\subset \R^m$, $\phi(p)=(x^1(p), \ldots, x^m(p))$ and $\bar U$ compact. In the coordinates $(x^1, \ldots, x^m)$, (up to the factor $1/\sigma$), the equation $\Delta u=0$ corresponds to  $\dive (\mathcal A(x, Du)=0$, where 
$\mathcal A\colon U\times \R^m\to \R^m$ is the map whose components are given by $\mathcal A^i(x,\omega):=\sigma(x)\frac 1 2\frac{\partial F^{*2}}{\partial \omega_i}(\omega)$ and $Du$ is the vector whose components are $(\partial_{x^1}u, \ldots,\partial_{x^m}u)$. Observe that $\mathcal A\in C^0(\bar U\times \R^m)\cap C^{k-1}(\bar U\times\R^m\setminus\{0\})$ (resp. $\mathcal A\in C^0(\bar U\times \R^m)\cap C^{\infty}(\bar U\times\R^m\setminus\{0\})$; $\mathcal A\in C^0(\bar U\times \R^m)\cap C^{\omega}(\bar U\times\R^m\setminus\{0\})$).
Thus, Finslerian harmonic functions are locally $\mathcal A$-harmonic in the sense of \cite{LiiSal14}. We notice that $\mathcal A$ satisfies the following  properties: there exists $C\geq 0$ such that
\begin{itemize}
	\item[1.]  for all $(x,\omega)\in U\times \R^m\setminus\{0\}$:
	\beq\|\mathcal A(x,\omega)\|+\|\partial_x\mathcal A(x,\omega)\|+\|\omega\|\,\|\partial_\omega \mathcal A(x,\omega)\|\leq C\|\omega\|;\label{uno}\eeq
	\item[2.] 
	for all $(x,\omega)\in U\times(\R^m\setminus \{0\})$ and all $h\in \R^m$:
	\beq\partial_\omega\mathcal A(x,\omega)[h,h]\geq \frac 1 C \|h\|^2;\label{due}\eeq
	\item[3.] for all $x\in U$ and $\omega_1, \omega_2\in \R^m$:
	\beq \big(\mathcal A(x,\omega_2)-\mathcal A(x,\omega_1)\big)(\omega_2-\omega_1)\geq \frac 1 C \|\omega_2-\omega_1\|^2;\label{tre}\eeq
\end{itemize}
in particular, \eqref{tre} comes from strong convexity of $F^{*2}$ on $TM\setminus 0$ (i.e. by \eqref{due}) and by a continuity argument when $\omega_2=-\lambda \omega_1$, for some $\lambda>0$.

By the theory of monotone operators or by a minimization argument based on the fact that $E$ satisfies the Palais-Smale condition
(see \cite[p.729-730]{GeShe01}) we have that for all $\varphi\in H^1(\Omega)$, there exists a minimum of $E$ on $\{\varphi\}\times H^1_0(\Omega)$ which is then a weak solution 
of \eqref{dirichlet}.

Now, as in \cite{GeShe01} or \cite{OhtStu09},  the following Proposition holds:
\bpr
Any weak  solution of \eqref{dirichlet}    belongs to $H^2_{\rm loc}(\Omega)\cap C^{1,\alpha}_{\rm loc}(\Omega)$, for some $\alpha\in (0,1)$. 
\epr
\bere The H\"older constant $\alpha$ in the above proposition depends on the  open relatively compact subset $\Omega'\subset \subset\Omega$ where $u$ is seen as a local weak solution of $\Delta u=0$, i.e. 
\[\frac 12 \int_\Omega \frac{\partial F^{*2}}{\partial \omega}(du)[d\eta] d\mu=0, \quad\quad \text{for all $\eta\in C^{\infty}_c(\Omega')$.}\]
\ere

From the above proposition and classical results for uniformly elliptic operators,  we can obtain higher regularity,   where $du\neq0$. In fact, if $du\neq 0$ on an open subset $U\subset \subset \Omega$  then, 
being $u\in H^2(U)$, the equation $\Delta u=0$ is equivalent to 
\beq \label{linear}\sum_{i,j=1}^m(g^*_{du})^{ij}\frac{\partial^2u}{\partial x^i\partial x^j}=-\sum_{i=1}^{m}\left(\frac{1}{ \sigma}\frac{\partial\sigma}{\partial x^i} \frac 1 2\frac{\partial F^{*2}}{\partial \omega_i}(du)+\frac 12 \frac{\partial^2F^{*2}}{\partial x^i\partial \omega_i}(du)\right), \quad \text{a.e. on $U$.}\eeq
This can be interpreted as a linear elliptic equation: 
\[a^{ij}(x)\partial_{x^ix^j}u=f(x),\]
where $a^{ij}(x):=(g^*_{du(x)})^{ij}$ and 
 \[f(x)=-\sum_{i=1}^{m}\left(\frac{1}{ \sigma(x)}\frac{\partial\sigma}{\partial x^i}(x) \frac 1 2\frac{\partial F^{*2}}{\partial \omega_i}(du(x))+\frac 12 \frac{\partial^2F^{*2}}{\partial x^i\partial \omega_i}(du(x))\right).\]
Thus, when  $k\geq 3$, the coefficients  $a^{ij}$ and $f$ are at least $\alpha$-H\"older continuous and then  $u\in C^{2,\alpha}(U)$; by a bootstrap argument, we  then get the  following proposition (see, e.g. \cite[Appendix J, Th. 40]{Besse08}):
\bpr\label{higherreg}
Let  $U\subset \subset\Omega$ be an open subset such that $du\neq 0$ on $U$ then  any weak solution  of \eqref{dirichlet}  belongs to $C^{k-1,\alpha}(U)$, for some $\alpha$ depending on $U$; moreover it is $C^\infty(U)$ (resp. $C^\omega(U)$) if $k=\infty$, (resp. if $M$ is endowed with an analytic structure,  $k=\omega$,  and $\sigma$ is also analytic).
\epr	
\section{Harmonic coordinates in Finsler manifolds}\label{harmonic}
Let us consider a smooth   atlas of the manifold  $M$ and a point $p\in M$.  Let $(U,\phi)$ be a chart of the atlas, with components  $(x^1, \ldots,x^m)$, such that $p\in U$, $\phi(p)=0$ and let us assume that the open ball $B_\epsilon (0)$, for some $\epsilon\in(0,1)$, is contained in $ \phi(U)$. \begin{proof}[Proof of Theorem~\ref{main}]
	Let $(u^i)_{i\in\{1,\ldots,m\}}$ be weak  solutions of the $m$ Dirichlet  problems:
	\[\begin{cases}
	\dive \big(\mathcal A(x, Du^i)\big)=0&\text{in } B_\epsilon(0)\\
	u^i=x^i&\text{on }\partial B_\epsilon(0)
	\end{cases}\]
	Following \cite[\S 3.9]{Taylor00} and \cite[Theorem 2.4]{LiiSal14}, we can re-scale the above problems by considering $\tilde u^i(\tilde x):=\frac{1}{\epsilon}u^i(\epsilon \tilde x)$ and $\mathcal A_\epsilon(\tilde x, \omega):=\mathcal A(\epsilon \tilde x, \omega)$, so that the problems are transferred on $B_1(0)$  with Dirichlet data $\tilde x^i:=\frac{x^i}{\epsilon}$. Notice that $\mathcal A_\epsilon$ satisfies \eqref{uno}--\eqref{tre} uniformly w.r.t. $\epsilon\in (0,1)$. Hence, there exists a solution $\tilde u^i\in H^1\big(B_1(0)\big)$ of
	\beq\begin{cases}
		\dive \big(\mathcal A_\epsilon(\tilde x, D\tilde u^i)\big)=0&\text{in } B_1(0)\\
		\tilde u^i=\tilde x^i&\text{on }\partial B_1(0)
	\end{cases}\label{direps}\eeq
	Moreover, there exists $\delta \in (0,1)$ such that $\|\tilde u^i\|_{C^{1,\alpha}\big(B_\delta(0)\big)}\leq C_1$, for all $i\in\{1,\ldots,m\}$, uniformly w.r.t. $\epsilon$  (notice  that, since $\alpha$ depends only on $m,\ C, \ \delta $ and $\|\tilde u^i\|_{H^1\big(B_1(0)\big)}$, which is uniformly bounded w.r.t. to $\epsilon$,   $\alpha$  is  independent of $\epsilon$ as well).
	
	Let us denote $\tilde x^i$ by $\tilde u^i_0$ and let $\tilde v^i:=\tilde u^i-\tilde u^i_0$. From \eqref{tre}, we have
	\[\int_{B_1(0)}|D\tilde v^i|^2d\tilde x\leq C\int_{B_1(0)}\big(\mathcal A_\epsilon(\tilde x,D\tilde u^i)-\mathcal A_\epsilon(\tilde x, D\tilde u^i_0)\big)(D\tilde v^i)d\tilde x.\]
	Recalling that $\tilde u^i$ solves \eqref{direps} and using $\tilde v^i$ as a test function, we obtain
	\baln
	\lefteqn{\int_{B_1(0)}\big(\mathcal A_\epsilon(\tilde x,D\tilde u^i)-\mathcal A_\epsilon(\tilde x, D\tilde u^i_0)\big)(D\tilde v^i)d\tilde x=}&\\
	&-\int_{B_1(0)}\mathcal A_\epsilon(\tilde x, D\tilde u^i_0)(D\tilde v^i)d\tilde x=\\
	&-\int_{B_1(0)}\big(\mathcal A_\epsilon(\tilde x,D\tilde u^i_0)-\mathcal A_0(0, D\tilde u^i_0)\big)(D\tilde v^i)d\tilde x
	\ealn
	where the last equality is a consequence of being $\tilde u^i_0$, $\mathcal A_0$-harmonic. As $\mathcal A$ is locally Lipschitz and $D\tilde u^i_0$ is a constant vector,  using also H\"older's inequality,  we then get 
	\[-\int_{B_1(0)}\big(\mathcal A_\epsilon(\tilde x,D\tilde u^i_0)-\mathcal A_0(0, D\tilde u^i_0)\big)(D\tilde v^i)d\tilde x\leq \epsilon C_2\left(\int_{B_1(0)}|D\tilde v^i|^2d\tilde x\right)^{1/2}\!\!\!\!\!\!;\]
	thus, $\|D\tilde v^i\|^2_{L^2\big(B_1(0)\big)}\leq \epsilon C_2 \|D\tilde v^i\|_{L^2\big(B_1(0)\big)}$, i.e.  $\|D\tilde v^i\|_{L^2\big(B_1(0)\big)}\leq \epsilon C_2$. Since there exists also a constant $C_4\geq 0$ such that
	\[\|D\tilde v^i\|_{C^{0,\alpha}\big(B_\delta(0)\big)}\leq  C_4,\]
	by \cite[Lemma A.1]{LiiSal14}) we get that there exists $\delta'\in (0,\delta)$ such that
	$\|D\tilde v^i\|_{L^\infty\big(B_{\delta'}(0)\big)}=o(1)$, as $\epsilon\to 0$, for all $i\in\{0,\ldots, m\}$. Therefore, if $\tilde \Phi(\tilde x):=\big(\tilde u^1(\tilde x), \ldots, \tilde u^m(\tilde x)\big)$, we have $\|D\tilde \Phi(0)-\mathrm{Id}\|=o(1)$, as $\epsilon\to 0$, and then  $D\Phi(0)=D\tilde \Phi(0)$  ($\Phi(x):=\big(u^1(x), \ldots, u^m(x)\big)$) is invertible, moreover up to considering a smaller $\delta'$, we have also that $Du^i(x)\neq 0$ for all $x\in     B_{\delta'}(0)$ and all $i\in\{1, \ldots, m\}$. Thus, from Proposition~\ref{higherreg},  $\Phi$ is a $C^{k-1,\alpha}$ (resp. $C^\infty$; analytic) diffeomorphism on $B_{\delta'}(0)$ and for $V_p:=\phi^{-1}(B_{\delta'}(0))$, we have that  $\Psi:=\Phi\circ \phi|V$ is a $C^{k-1, \alpha}$ (resp. $C^\infty$; analytic) diffeomorphism whose components $u^i\circ\phi$ are harmonic. 
\end{proof}
Harmonic coordinates were successfully used  by M. Taylor \cite{Taylor06} in a new  proof of the Myers-Steenrod theorem about regularity of isometries between  Riemannian manifolds (including the case when the metrics are only H\"older continuous). In the Finsler setting, Myers-Steenrod theorem has been obtained with different methods in \cite{DenHou02, AraKer14} for smooth (and strongly convex) Finsler metrics and in \cite{MatTro17} for H\"older continuos Finsler metrics. We show here that harmonic coordinates can be used to prove the Myers-Steenrod  in the Finsler setting too, provided that the Finsler metrics are smooth enough.

Let $(M_1, F_1,\mu_1)$ and   $(M_2, F_2,\mu_2)$ be two oriented Finsler manifold of the same dimension $m$ endowed with the volume forms $\mu_1$ and $\mu_2$. 
Let us assume that 
$\mu_1$ and $\mu_2$ are locally Lipschitz  with locally bounded  differential meaning that in the local expressions of $\mu_1$ and $\mu_2$, $\mu_1=\sigma_1dx^1\wedge\ldots\wedge d x^m$, $\mu_2=\sigma_2dx^1\wedge\ldots\wedge d  x^m$,  $\sigma_1$ and $\sigma_2$ are Lipschitz functions and their derivatives (which are defined a.e. by Rademacher's theorem),  $\partial_{x^i}\sigma_1$ and $\partial_{x^i}\sigma_2$, for each $i\in\{1,\ldots, m\}$,  are  $L^\infty$ functions. Let $d_i$, $i=1,2$, be the (non-symmetric) distances associated to $F_i$.  Let $\mathcal I:M_1\to M_2$ be a distance preserving bijection. Clearly,  $\mathcal I$ is an isometry  of the symmetric distance  $\tilde d_i(x_1,x_2):=d_i(x,y)+d_i(y,x)$ thus, in particular, it is a bi-Lipschitz map (i.e. it is Lipschitz with Lipschitz inverse) w.r.t. the distances $\tilde d_i$ and it is locally Lipschitz w.r.t. the distances associated to any Riemannian metric on $M_1$ and $M_2$. Hence, $\mathcal I$ and its inverse are differentiable a.e. on $M_1$ and, respectively,  $M_2$.    In the next lemma we deal with  the relations existing between the Finsler metrics $F_1$ and $F_2$,  the inverse maps of their Legendre maps $\ell_1$ and $\ell_2$, their co-Finsler metrics $F^*_1$ and $F^*_2$, in presence of an isometry $\mathcal I$. For a fixed $x$ in $M_1$ or $M_2$, let us denote by $\mathcal J_{i, x}$, $i=1,2$, the diffeomorphisms between $T_xM_i$ and $T_x^*M_i$, given by $\mathcal J_{i,x}(v):=\frac 1 2\frac{\partial}{\partial v}F^2_i(x,v)[\cdot]$.
\begin{Lemma}\label{lemma}
	Let $\mathcal I$ be a distance preserving bijection between the Finsler manifolds $(M_1, F_1, \mu_1)$ and $(M_2, F_2,\mu_2)$. Then, for a.e. $x\in M_1$, we have:
	\begin{itemize}
		\item[(a)] $F_1=\mathcal I^*(F_2)$,  (i.e., $F_1(x,v)=F_2(\mathcal I(x), d\mathcal I(x)[v])$);
		\item[(b)] $\ell_1^{-1}(x,\omega)=\big(x,d\mathcal I^{-1}(\mathcal I(x))[\mathcal J^{-1}_{2,\mathcal I(x)}(\omega\circ d\mathcal I^{-1})]\big)$;
		\item[(c)] $F_1^*= \mathcal I^*(F^*_2)$,  (i.e., $F^*_1(x,\omega)=F^*_2\big(\mathcal I(x), \omega\circ d\mathcal I^{-1}\big)$).
	\end{itemize}	
\end{Lemma}
\begin{proof}
	It is well-known that any Finsler metric $F$ on a manifold $M$ can be computed by using the associated  distance $d$ as 
	$F(x,v)=\displaystyle\lim_{t\to 0^+}\dfrac1 t d\big(\gamma(0), \gamma(t)\big)$, where $\gamma$  is a smooth curve on $M$ such that $\gamma(0)=x$ and $\dot\gamma(0)=v$; hence, (a) immediately follows from this property and the fact that $\mathcal I$ is a distance preserving map. From $(a)$, we get
	\bmln
	\mathcal J_{1,x}(v)=\frac 1 2 \frac{\partial}{\partial v}\Big(F^2_2\big(\mathcal I(x), d\mathcal I(x)[v]\big)\Big)=\frac 1 2\frac{\partial}{\partial w}F^2_2\big(\mathcal I(x), d\mathcal I(x)[v]\big)\big[d\mathcal I(x)[\cdot]\big]\\=\frac 1 2\mathcal J_{2,\mathcal I(x)}(d\mathcal I(x)[v])\big[d\mathcal I(x)[\cdot]\big],\emln 
	and then  we deduce $(b)$. Finally, $(c)$ follows from
	\bmln 
	F_1^*(x,\omega)=F_1\big(\ell_1^{-1}(x,\omega)\big)=F_2\Big(\mathcal I(x), d\mathcal I(x)\Big[d\mathcal I^{-1}(\mathcal I(x))[\mathcal J^{-1}_{2,\mathcal I(x)}(\omega\circ d\mathcal I^{-1})]\Big]\Big)\\=F_2\big(\mathcal I(x), \mathcal J^{-1}_{2,\mathcal I(x)}(\omega\circ d\mathcal I^{-1})\big)=F^*_2(\mathcal I(x), \omega\circ d\mathcal I^{-1})\emln
\end{proof}
\bpr\label{MyersSteenrod}
Let	$(M_1, F_1, \mu_1)$ and $(M_2, F_2,\mu_2)$ be two oriented Finsler manifolds, where  $\mu_1$ and $\mu_2$ are  locally Lipschitz with locally bounded differential (in the sense specified above) volume forms. Let $\mathcal I\colon M_1\to M_2$ be a distance preserving bijective map; if  $\mu_1=\mathcal I^*(\mu_2)$ (i.e. locally $\sigma_1=|\mathrm{Jac}(\mathcal I)|\sigma_2\circ\mathcal I $, where $\mathrm{Jac}(\mathcal I)$ is the Jacobian of $\mathcal I$) and $F_i$, $i=1,2$ are at least $C^3$ on $TM_i\setminus 0$, then $\mathcal I$ is a $C^{1}$ diffeomorphism.
\epr
\begin{proof}
	Being $\mathcal I$ a bi-Lipshitz map, it is enough to prove that $\mathcal I$ is  locally $C^1$ with locally $C^1$ inverse. Under the assumptions on $F_i$ and $\mu_i$, $i=1,2$, we have that   \eqref{uno}--\eqref{tre} hold and then, given an open relatively compact subset $\Omega_1\subset M_1$ with Lipschitz boundary, we have the existence of a minimum of the functional $E_1$ on $\{\varphi\}\times H^1_0(\Omega_1)$, for any $\varphi\in H^1(\Omega_1)$; clearly, the same existence of minima holds for $E_2$ on analogous  subsets $\Omega_2\subset\subset M_2$. Moreover, the same assumptions on $F_i$ and $\mu_i$ ensure that such  minima are  locally weak  harmonic and, arguing as in \cite[Th. 4.6 and Th. 4.9]{OhtStu09}, they are $H^2$ and  $C^{1,\alpha}$ on subsets $U_1\subset\subset \Omega_1$ (resp. $U_2\subset\subset \Omega_2$). Therefore, Theorem~\ref{main}, still holds under the assumptions of Proposition~\ref{MyersSteenrod}, and it provides $C^{1,\alpha}$ diffeomorphisms  $\Psi_i$, $i=1,2$, whose components  are harmonic. Let now $p\in M_1$ and $(U_2,\Psi_2)$ be a chart of harmonic coordinates of $(M_2, F_2,\mu_2)$ centred at $\mathcal I (p)$, $\Psi_2=(u_2^1,\ldots,u_2^m)$. Let us show that, for each $j\in\{1,\ldots,m\}\}$, $u_1^j:=u_2^j\circ\mathcal I$ is weakly harmonic on  $U_1=\phi^{-1}(U_2)$. First we notice that for any function $u\in H^1(U_2)$, $u\circ\mathcal I\in H^1(U_1)$, because $\mathcal I$ is Lipschitz; moreover, from $(c)$ of Lemma~\ref{lemma} and the change of variable formula for integrals under  bi-Lipschitz transformations, we have
	\bmln E_1 (u\circ \mathcal I|_{U_1})=
	\frac 12 \int_{U_1}F^{*2}_2\Big(\mathcal I(x), \big(du\circ d\mathcal I\circ d\mathcal I^{-1}\big)\big(I(x)\big)\Big)\mathcal I^*(d\mu_2)\\	=\frac 12 \int_{U_2}F^{*2}_2\big(x, du(x)\big)d\mu_2.\emln
	Hence, being $u^j_2$ a minimum of $E_2$ on $\{u^j_2\}\times H^1_0(U_2)$, we deduce that $u^j_1$ is a minimum 
	of $E_1$ on $\{u^j_1\}\times H^1_0(U_1)$ and then it is  a weakly harmonic function. Therefore, $\Psi_1:=(u_1^1,\ldots, u^m_1)$ is a $C^{1, \alpha}$ diffeomorphism and $\mathcal I|_{U_1}=\Psi_2^{-1}\circ\Psi_1$ and its inverse are both  $C^{1,\alpha}$ map as well.  
\end{proof}
\bere
If, for each $i\in\{1,2\}$, $F_i$ is  of class $C^{k+1}$ on $TM\setminus 0$, with $k\geq 3$, and $\mu_i$ is   of class $C^{k-1}$, as in  Theorem~\ref{main} (recall, in particular, \eqref{linear}),  we can deduce  that $\mathcal I$ is a $C^{k-1}$ diffeomorphism provided that $\mu_1$ and $\mu_2$ are related by $\mu_1=\mathcal I^*(\mu_2)$.
\ere

\section{Regularity results for Berwald metrics}\label{Berwald}
Let us recall the following result from \cite{DeTKaz81} which gives optimal  regularity  of a Riemannian metric in harmonic coordinates in connection with the regularity of the Ricci tensor (the meaning of ``optimal'' here is illustrated  in all its facets in \cite{DeTKaz81}).
\begin{Theorem}[Deturck - Kazdan]\label{detkaz}
	Let $h$ be a $C^2$ Riemannian metric  and 
	$\mathrm{Ric}(h)$ be its Ricci tensor.
	If in harmonic coordinates of $h$, $\mathrm{Ric}(h)$ is  of class $C^{k,\alpha}$, for $k\geq 0$, (resp. $C^\infty$; $C^\omega$) then in these coordinates $h$ is of class $C^{k+2,\alpha}$ (resp. $C^\infty$; $C^\omega$).
\end{Theorem}	
Let us consider now  a Finsler manifold $(M,F)$ such that $F$ is $C^4$ on $TM\setminus 0$. A role similar to the one of the  Ricci tensor in the result above will be played by the {\em Riemann curvature} of $F$. This is a family of linear transformations of the tangent spaces defined in the following way (see \cite[p. 97]{Shen01}):
let $G^i(x,y)$, $y\in T_xM\setminus\{0\}$, $x\in M$   be the spray coefficients of $F$:
\[G^i(x,y):=\frac 1 4g^{ij}(x,y)\left(\frac{\partial^2F^2}{\partial x^k\partial y^j}(x,y)y^k-\frac{\partial F^2}{\partial x^j}(x,y)\right),\]
where $g^{ij}(x,y)$ are the components of the inverse of the matrix representing the fundamental tensor $g$ at the point $(x,y)\in TM\setminus 0$. 

Let
\bmln R^i_k(x,y):=2\frac{\partial G^i}{\partial x^k}(x,y)-y^m\frac{\partial^2 G^i}{\partial x^m\partial y^k}(x,y)\\+2G^m(x,y)\frac{\partial^2 G^i}{\partial y^m\partial y^k}(x,y)-\frac{\partial G^i}{\partial y^m}(x,y)\frac{\partial G^m}{\partial y^k}(x,y).\emln
As above we will omit the explicit dependence on $x$, by writing simply $R^i_k(y)$. The Riemann curvature of $F$ at $y\in T_xM\setminus \{0\}$ is then the linear map $\mathbf R_y:T_xM\to T_xM$ given by
$\mathbf R_y:=R^i_k(y)\partial_{x^i}\otimes dx^k$. It can be shown (see \cite[Eqs. (8.11)-(8.12)]{Shen01} that 
\[R^i_k(y)=R^i_{jkl}(y)y^jy^l,\] 
where $R^i_{jkl}$ are the components of the $hh$ part of the curvature $2$-forms of the Chern connection, which are equal, in natural local coordinate on $TM$, to
\beq R^i_{jkl}(y):=\frac{\delta \Gamma^i_{jl}}{\delta x^k}(y)-\frac{\delta \Gamma^i_{jk}}{\delta x^l}(y)+\Gamma^m_{jl}(y)\Gamma^i_{mk}(y) -\Gamma^m_{jk}(y)\Gamma^i_{ml}(y),\label{Rijkl}\eeq
$\Gamma^i_{jk}$ being  the components of the Chern connection and $\frac {\delta }{\delta x^i}$ be the vector field on $TM\setminus 0$ defined by $\frac {\delta }{\delta x^i}:=\frac{\partial}{\partial x^i}-N^m_i(y)\frac{\partial }{\partial y^m}$, where $N^m_i(y):=\frac{1}{2}\frac{\partial G^m}{\partial y^i}(y)$. 

Finally, let us introduce the {\em Finsler Ricci scalar} as the contraction of the Riemann curvature $R(y):=R^i_i(y)$ (see \cite[Eq. (6.10)]{Shen01}).

We recall that, if $F$ is the norm of a Riemannian metric $h$ (i.e. $F(y)=\sqrt{h(y,y)}$) then  the components of the Chern connection coincide with those of the Levi-Civita connection, so they do not depend on $y$ but only on $x\in M$,  and   the functions  $R^i_{jkl}$  are then equal to the components of the standard Riemannian curvature tensor of $h$.

Let us also recall  (see \cite[p. 85]{Shen01}) that to any Finsler manifold  $(M,F)$ we can associate a canonical  {\em covariant derivative} of a vector field $V$ on $M$ in the direction $y\in TM$ defined, in local coordinates, as
\[D_yV:=\left(dV^i(y)+V^j(x)N^i_j(y)\right)\partial_{x^i}|_x,\quad \quad x=\pi(y)\]
and extending it as $0$ if $y=0$.

There are several equivalent ways to introduce Berwald metrics (see e.g.   \cite{SzLoKe11});  we will say  that a Finsler metric is said {\em Berwald} if the nonlinear connection $N^i_j$ is actually  a linear connection on $M$; thus, $N^i_j(y)=\Gamma^i_{jk}(x)y^k$   (see \cite[prop. 10.2.1]{BaChSh00}), so that  the  components  of the Chern connection do not depend on $y$; from \eqref{Rijkl},  the same holds for the components of the Riemannian curvature tensor $R^i_{jkl}$.

From a result by Z. Szab\'o \cite{Szabo81}, we know that there exists a Riemannian metric $h$ such that its Levi-Civita connection is equal to the Chern connection of $(M,F)$.  Actually such a Riemannian metric is not unique and  different ways to construct one do exist; in particular a branch of these methods is based on averaging over the indicatrixes (or, equivalently, on the unit balls) of the Finsler metric $S_x:=\{y\in T_xM:F(y)=1\}$, $x\in M$   (see the nice review   \cite{Crampi14}); moreover, the fundamental tensor of $F$ can be used in this averaging procedure as shown first by  C. Vincze \cite{Vincze05} and then, in a slight different way, in \cite{Matvee09a} (based on  \cite{MaRaTZ09}) and in other manners also in   \cite{Crampi14}. In particular,  as described in \cite{Crampi14}, the Riemannian metric obtained in \cite{MaRaTZ09} is given, up to a constant conformal factor in the Berwald case, by 
\beq h_x(V_1,V_2):=\frac{\int_{S_x} g(x,y)[V_1,V_2]d\lambda}{\int_{S_x}d\lambda},\label{riemannian}\eeq
where $d\lambda$ denotes the measure induced on $S_x$ (seen as an hypersurface on $\R^m\cong T_xM$) by the Lebesgue measure on $\R^m$.

\bpr\label{regBerwald}
Let $(M,F)$ be a  Berwald  manifold such that $F$ is  a  $C^4$ function on $TM\setminus 0$. Assume that the Finsler Ricci scalar  $R$ of $F$ is  of class $C^{k,\alpha}$, $k\geq 0$, (resp. $C^\infty$; $C^\omega$)  on $TW\setminus 0$, for some open set $W\subset M$.  Then the Riemannian metric $h$ in \eqref{riemannian} is of class  $C^{k+2,\alpha}$ (resp. $C^\infty$; $C^\omega$) in a  system of harmonic coordinates $(U,\phi)$, $U\subset W$, of the same    metric. 
\epr
\begin{proof}
	Being $F$ of class $C^4$ on $TM\setminus 0$,  the partial derivatives of $g(x,y)$, up to the second order, exist on $TM\setminus 0$ and are continuous;  
	for each $y\in S_x$, $\frac{\partial F^2(y)}{\partial y}[y]=2F^2(y)\neq 0$, thus $1$ is a regular value of the function $F^2$ and  the indicatrix bundle $\{(x,y)\in TM: F(x,y)=1\}$ is a $C^4$ embedded hypersurface in $TM$.  Thus,  both the area  of $S_x$ and the numerator in \eqref{riemannian} are  $C^2$ in $x$ and then $h$ is a $C^2$ Riemannian metric on $M$.
	From  \eqref{Rijkl} and the fact that $F$ is Berwald, the components $R^i_{jkl}$ are  equal to the ones of the Riemannian curvature tensor of $h$ and then  we have \beq\label{rich}\mathrm{Ric}(h)_{\alpha\beta}(x)=R^m_{\alpha m\beta}(x)=\frac 1 2 \frac{\partial^2}{\partial y^\alpha\partial y^\beta}\left (R^m_{jml}(x)y^jy^l\right)=\frac 1 2 \frac{\partial^2 R}{\partial y^\alpha\partial y^\beta}(x,y),\eeq  
	moreover  $R$ is quadratic in the $y^j$ variables, i.e.  $R(x,y)=R^m_{jml}(x)y^jy^l$, and  its second vertical derivatives  $\frac{\partial^2 R}{\partial y^\alpha\partial y^\beta}$ being independent  of $y^j$,  are $C^{k,\alpha}$ (resp. $C^\infty$; $C^\omega$) functions on $W$.
	Thus, the result follows from \eqref{rich} and Theorem~\ref{detkaz}.	
\end{proof}
\bere
Clearly, an analogous result holds for any $C^2$ Riemannian metric such that its Levi-Civita connection is equal to the canonical connection of the Berwald metric as the Binet-Legendre metric in  \cite{MatTro12}.
\ere
\bere
Under the assumptions of Proposition~\ref{regBerwald}, we get that components of the Chern connection of the Berwald metric $F$ are $C^{k+1,\alpha}$ (resp. $C^\infty$; $C^\omega$) in harmonic coordinates of the metric $h$.  In particular the geodesic vector field 
$y^i\partial_{x^i}-\Gamma^i_{jk}y^jy^k\partial_{y^i}$ is $C^{k+1,\alpha}$ (resp. $C^\infty$; $C^\omega$) in the corresponding natural coordinate system of $TM$.
\ere
\bere\label{horizontal}
Other  notions of Finslerian  Laplacian, which take into account the geometry of the tangent bundle more than the nonlinear Finsler Laplacian, could be considered in trying to obtain some regularity result for the fundamental tensor without averaging. Natural candidates are the horizontal Laplacians studied  in \cite{BaoLac98} which in the Berwald case   are equal (up to a minus  sign) to 
\[\Delta_H f= \frac{1}{\sqrt g}\frac{\delta}{\delta x^i}\left(\sqrt{g}g^{ij}\frac{\delta f}{\delta x^j}\right),\]
where  $f$ is a  smooth function defined on some open subset of $TM$ and $\sqrt{g}:=\sqrt{\det g(x,y)}$. We notice that this is also the definition of the horizontal Laplacian of  any Finsler metric given \cite{DraLar89}; moreover  for a $C^2$ function $f:M\to \R$, $\Delta_Hf$ is equal to the $g$-trace of the Finslerian Hessian of $f$,
$\mathrm{Hess} f:=\nabla(df)$, where $\nabla$ is the Chern connection. In natural local coordinates of $TM$,  for each $(x,y)\in TM\setminus 0$ and all $u,v\in T_xM$,  $\mathrm{Hess}f(x,y)[u,v]$ is given     by $(\mathrm{Hess}f)_{ij}(x,y)u^iv^j=\frac{\partial^2f}{\partial x^i\partial x^j}(x)u^iv^j-\frac{\partial f}{\partial x^k}(x)\Gamma^k_{\,\,ij}(x,y)u^iv^j$ (see e.g. \cite{CaJaMa10}), hence the  $g$-trace of $\mathrm{Hess}f$ is equal to $g^{ij}(x,y)\frac{\partial^2f}{\partial x^i\partial x^j}(x)-\frac{\partial f}{\partial x^k}(x)\Gamma^k(x,y)$, where $ \Gamma^k:=g^{ij}\Gamma^k_{ij}$. This expression is equivalent to  
\beq\label{hessf}
g^{ij}(x,y)(\mathrm{Hess}f)_{ij} (x,y)=g^{ij}(x,y)\frac{\delta^2f}{\delta x^i\delta x^j}(x)-\frac{\delta f}{\delta x^k}(x)\Gamma^k(x,y)
\eeq
because $\frac{\delta f}{\delta x^k}=\frac{\partial f}{\partial x^k}$ for a function defined  on $M$.  For a $C^2$ function $f$ on $TM\setminus 0$,  taking into account that $\frac 12 g^{lm}\frac{\delta g_{lm}}{\delta x^i}=\Gamma^l_{li}$ and $g^{ij}\Gamma^l_{li}+	\frac{\delta g^{ij}}{\delta x^i}=-g^{pq}\Gamma^j_{pq}$, we have 
\baln \Delta_H f&=\frac 12g^{lm}\frac{\delta g_{lm}}{\delta x^i}g^{ij}  \frac{\delta f}{\delta x^j}+\frac{\delta g^{ij}}{\delta x^i}  \frac{\delta f}{\delta x^j}+
g^{ij}\frac{\delta^2f}{\delta x^i\delta x^j}\\
&= g^{ij}\Gamma^l_{li}\frac{\delta f}{\delta x^j}+	\frac{\delta g^{ij}}{\delta x^i}  \frac{\delta f}{\delta x^j}+
g^{ij}\frac{\delta^2f}{\delta x^i\delta x^j}\\
&=-g^{hk}\Gamma^j_{hk}\frac{\delta f}{\delta x^j}+
g^{ij}\frac{\delta^2f}{\delta x^i\delta x^j},	
\ealn
which coincides with \eqref{hessf} when $f$ is a  function defined on $M$.

In particular, if $f$ is a $\Delta_H$-harmonic coordinate on $M$ then  
\[0=\frac{1}{\sqrt g}\frac{\delta}{\delta x^i}\left(\sqrt{g}g^{ij}\right)=g^{hk}\Gamma^j_{hk}= \Gamma^j,\]
which is analogous to \eqref{gammah}.
Anyway,   ellipticity (and quasi diagonality),  that hold  in the Riemannian case for  the system $R_{ij}:=R^m_{imj}= A_{ij}$, where $ A_{ij}$ are $C^{k,\alpha}$ (resp. $C^\infty$; $C^\omega$) functions,  would be spoiled,  in $\Delta_H$-harmonic coordinates  of a Berwald metric, by the presence of second order terms of the type $g^{rs}N^l_s\frac{\partial^2 g_{ij}}{\partial x^r y^l}$.  
\ere


\end{document}